\documentclass[11pt]{amsart}
\usepackage{enumerate}
\usepackage[all]{xy}
\usepackage{amsmath,amssymb}
\usepackage{graphicx}
\usepackage[mathscr]{euscript}

\newtheorem{theorem}{Theorem}[section]
\newtheorem{proposition}[theorem]{Proposition}
\newtheorem{lemma}[theorem]{Lemma}
\newtheorem{corollary}[theorem]{Corollary}

\newtheorem{definition}[theorem]{Definition}
\newtheorem{introtheorem}{Theorem}

\theoremstyle{remark}
\newtheorem{example}[theorem]{Example}
\newtheorem{remark}[theorem]{Remark}

\newcommand{\secref}[1]{\S\ref{#1}}
\newcommand{\thmref}[1]{Theorem~\ref{#1}}
\newcommand{\propref}[1]{Proposition~\ref{#1}}

\newcommand{\corref}[1]{Corollary~\ref{#1}}

\newcommand{\examref}[1]{Example~\ref{#1}}
\newcommand{\introthmref}[1]{Theorem~\ref{#1}}

\def\cat{\sf{cat}}
\def\cupp{\cup}
\def\dim{\mathrm{dim}}

\def\Z{\mathbb Z}

\def\R{\mathbb R}

\def\U{\mathcal U}
\def\V{\mathcal V}
\def\W{\mathcal W}

\def\RP{\mathbb {RP}}

\def\cD{\mathcal D}

\def\TC{{\mathsf{TC}}}
\def\UTC{\widetilde{\TC}}

\def\secat{{\sf{secat}}}

\newcommand{\chd}{\mathrm{cd}}

\newcommand{\defref}[1]{Definition~\ref{#1}}

\newcommand{\cd}{{\rm {cd}}}
\newcommand{\tsc}{ {\widetilde {\sf secat}(E\stackrel p\to \overline X\stackrel q\to X)}}

\begin{document}

\title[An Upper Bound for Topological Complexity]{An Upper Bound for Topological Complexity}

\author{Michael Farber}
\address{School of Mathematical Sciences \\
Queen Mary University of London\\
London, E1 4NS\\
United Kingdom}
\email{m.farber@qmul.ac.uk}
\author{Mark Grant}
\address{Institute of Pure and Applied Mathematics \\
University of Aberdeen\\
Aberdeen AB24 3UE\\
United Kingdom}
\email{mark.grant@abdn.ac.uk}
\author{Gregory Lupton}
\address{Department of Mathematics\\
Cleveland State University\\
Cleveland OH 44115  \\
U.S.A.}
\email{g.lupton@csuohio.edu}
\author{John Oprea}
\thanks{This work was partially supported by a grant from the Simons Foundation:
(\#244393 to John Oprea).}
\address{Department of Mathematics\\
Cleveland State University\\
Cleveland OH 44115  \\
U.S.A.}
\email{j.oprea@csuohio.edu}

\subjclass[2000]{Primary 55M30; Secondary 55P99}
\keywords{topological complexity, Lusternik-Schnirelmann category}

\begin{abstract}
In \cite{FGLO},  a new approximating invariant $\TC^\cD$ for topological complexity was introduced
called $\cD$-topological complexity. In this paper, we explore more fully the properties of $\TC^\cD$ and
the connections between $\TC^\cD$ and invariants of Lusternik-Schnirelmann type. We also introduce a new
$\TC$-type invariant $\UTC$ that can be used to give an upper bound for $\TC$,
$$\TC(X)\le \TC^\cD(X) + \left\lceil \frac{2\dim X -k}{k+1}\right\rceil,$$
where $X$ is a finite dimensional simplicial complex with $k$-connected universal cover $\tilde X$.
The above inequality is a refinement of an estimate given by Dranishnikov \cite{Dr3}.
\end{abstract}

\maketitle

\section{Introduction}\label{sec:intro}
Topological complexity $\TC(X)$ is a numerical homotopy invariant introduced by Farber \cite{Far}.
As well as being of intrinsic interest to homotopy theorists, its study is motivated by topological aspects of the motion
planning problem in robotics. The number $\TC(X)$ gives a quantitative measure of the
`navigational complexity' of $X$, when viewed as the configuration space of a mechanical system. Topological complexity
is a close relative of the Lusternik--Schnirelmann category $\cat(X)$ (see \cite{CLOT}), although the two notions are independent.

Recall that  $\cat(X)$ is the smallest $n$ such that $X$ admits an open covering $\{ U_0, \ldots, U_n \}$ by $(n+1)$ sets, each of which is contractible in $X$.
The \emph{sectional category} of a fibration $p \colon E \to B$, denoted
by $\secat(p)$, is the smallest number $n$ for which there is an open covering $\{ U_0, \ldots, U_n \}$ of $B$ by $(n+1)$ open
sets, for each of which there is a continuous local section $s_i \colon U_i \to E$ of  $p$; that is,  $p\circ s_i = j_i \colon U_i \to B$,
where $j_i$ denotes the inclusion.

Let $X^I$ denote the space of (free) paths in a space $X$.  There is a fibration
$$\pi_X\colon X^I \to X\times X,$$
which evaluates a path at initial and final points: for $\alpha \in X^I$, we have $\pi_X(\alpha) = \big(\alpha(0), \alpha(1)\big)$.
We define the \emph{topological complexity} $\TC(X)$ of
$X$ to be the sectional category $\secat\big( \pi_X\big)$ of this fibration.  That is, $\TC(X)$ is the smallest number $n$ for
which there is an open cover $\{ U_0, \ldots, U_n \}$ of $X \times X$ by $(n+1)$ open sets, for each of which there is a continuous
section $s_i \colon U_i \to X^I$ of  $\pi_X$,  $\pi_X\circ s_i = j_i \colon U_i \to X \times X$, where $j_i$ denotes the inclusion.

Just as LS category is very difficult to compute, so also is topological complexity. Indeed, it is usually the case for both invariants that
lower and upper bounds are derived. The fundamental such bounds are the following.

\begin{introtheorem}\label{introthm:bounds}
The following bounds hold \cite{Far}:
$$\cat(X) \leq \TC(X) \leq \cat(X \times X).$$
\end{introtheorem}

When $X=K(\pi,1)$ is aspherical the topological complexity $\TC(X)$ depends only on $\pi$ and we may write $\TC(X)=\TC(\pi)$.
It is easy to see (using the Eilenberg - Ganea theorem \cite{EG}) that  $\TC(\pi)$ is finite if and only if there exists a finite dimensional $K(\pi,1)$.
The following estimate was obtained in \cite{Dr3}.

\begin{introtheorem}\label{introthm:Dranish}
Let $X$ be a finite CW-complex with fundamental group $\pi$. Then
\begin{eqnarray}\label{dr}\TC(X) \leq \TC(\pi) + \dim(X).\end{eqnarray}
\end{introtheorem}

\noindent Of course, this estimate is only meaningful when $\TC(\pi)$ is finite. This rules out, for instance, any group $\pi$
having torsion. The inequality (\ref{dr}) is effective when the group $\pi$ has \lq\lq small\rq\rq\, cohomological dimension, say, is trivial, a free group or a surface group etc.

In this paper, we will refine the estimate of Theorem \ref{introthm:Dranish}
by using an invariant $\TC^\cD(X)$ defined in \cite{FGLO} as well as a new invariant $\UTC(X)$ defined in this article
(see \examref{coincide}).

\begin{introtheorem}\label{introthm:refine}
Let $X$ be a CW-complex with fundamental group $\pi$. Then
\begin{eqnarray}\label{three}\TC(X) \leq \TC^\cD(X) + \UTC(X).
\end{eqnarray}
Moreover, if the universal cover $\tilde X$ is $k$-connected, then
\begin{eqnarray}
\TC(X)  \le \TC^\cD(X) + \left\lceil \frac{2\,\dim X -k}{k+1}\right\rceil.
\end{eqnarray}
\end{introtheorem}

The definitions of the invariants $\TC^\cD(X)$ and $\UTC(X)$ are  given in \S 2 and \S 3. Note, however, that in contrast to $\TC(\pi)$,
$\TC^\cD(X)$ is \emph{finite} whenever $\TC(X)$ is. Furthermore, it equals $\TC(\pi)$ for $X=K(\pi,1)$. In general one has an inequality
\begin{eqnarray}\label{one}\TC^\cD(X)\le \TC(\pi), \quad \mbox{where}\quad \pi=\pi_1(X).\end{eqnarray}
For the invariant $\UTC(X)$, we will see that it is a special case of a new type of sectional category invariant $\tsc$ associated to a
covering map $q\colon \overline X \to X$ and a fibration $p\colon E \to \overline X$. The invariant $\tsc$ and its properties will be
explored in \secref{sec:fibcov}. As a special case of \thmref{main2}, we will also see that
\begin{eqnarray}\label{two}\UTC(X) \leq \dim (X).\end{eqnarray}
Hence both terms on the right-hand side of inequality (\ref{three}) are dominated by the corresponding terms of the right-hand side of (\ref{dr}).
We give specific examples when (\ref{three}) is sharper than (\ref{dr}). We note, however, that Dranishnikov actually proves a stronger inequality than (\ref{dr})
using his notion of \emph{strongly equivariant topological complexity} $\TC^*_\pi$ and uses the more ``practical'' inequality (\ref{dr}) because $\TC^*_\pi$ is very
difficult to compute. We show in \propref{prop:setc} that our invariant $\UTC(X)$ is equal to $\TC^*_\pi(\tilde X)$ so that in the general case (\ref{three})
is a refinement because of the first term $\TC^\cD(X)$ alone. Nevertheless, using $\UTC$ provides not only a much simpler proof of the general upper
bound  in \introthmref{introthm:refine} but also allows the generalization to the improved connectivity-dimension upper bound in \introthmref{introthm:refine} as well.



\section{The $\cD$-Topological Complexity}\label{sec:tc}
Let us recall from \cite{FGLO} the following definition.

\begin{definition}\label{def:tcd} Let X be a path-connected topological space with fundamental group
$\pi = \pi_1(X; x_0)$. The $\cD$-topological complexity, $\TC^\cD(X)$, is defined as the minimal number
k such that $X \times X$ can be covered by $k +1$ open subsets $$X \times X = U_0 \cupp U_1 \cupp \cdots \cupp U_k$$ with
the property that for any $i = 0, 1, 2, \ldots, k$ and for every choice of the basepoint $u_i \in U_i$,
the homomorphism $\pi_1(U_i; u_i) \to \pi_1(X \times X; u_i)$ induced by the inclusion $U_i \to X \times X$ takes
values in a subgroup conjugate to the diagonal $\Delta \subset \pi\times\pi$.
\end{definition}
\noindent Note that the letter $\cD$ in the notation $\TC^\cD(X)$ stands for the \lq\lq diagonal\rq\rq.

Here we mention that for each point $u_i\in X\times X$, there is an isomorphism
$\pi_1(X \times X; u_i) \to \pi_1(X \times X; (x_0; x_0)) = \pi\times\pi$
determined uniquely up to conjugation, and the diagonal inclusion $X \to X \times X$ induces
the inclusion $\pi \to \pi\times\pi$ onto the diagonal $\Delta$.

Recall that a topological space $X$ admits a universal cover if it is connected, locally path connected and semi-locally simply connected. Since these conditions are preserved under products, it then follows that $X\times X$ admits a universal cover. In particular, $X\times X$ admits a universal cover whenever $X$ is a locally finite cell complex.


\begin{proposition}\label{cor:covspace} Let $X$ be a connected, locally path connected and semi-locally simply connected topological space with fundamental group $\pi =
\pi_1(X; x_0)$. Let $q \colon \widehat{X \times X} \to X \times X$ be the connected covering space corresponding to the
diagonal subgroup $\Delta \subset \pi\times\pi = \pi_1(X \times X; (x_0; x_0))$.
Then the $\cD$-topological complexity satisfies
$$\TC^\cD(X)=\secat(q);$$
that is, $\TC^\cD(X)$ equals the sectional category of the covering $q$.
\end{proposition}

\begin{proof}
If $U \subset X \times X$ is an open subset, then a partial section $U \to \widehat{X\times X}$ of $q$ gives a
factorisation of the homomorphism  of fundamental groups induced by the inclusion $U\to X\times X$ through the diagonal. Now,
since $q$ is a covering, for an open subset $U \subset X \times X$, the condition that the induced map $\pi_1(U; u) \to
\pi_1(X \times X; u)$ takes values in a subgroup conjugate to the diagonal $\Delta$ implies that
$q$ admits a continuous section over $U$. Using this remark the result follows
by comparing the definitions of $\TC^\cD(X)$ and of sectional category.
\end{proof}

\begin{example}
For a path-connected space $X$ one has $\TC^\cD(X)=0$ if and only if $X$ is simply connected; this follows directly from the definition. In particular we have
$\TC^\cD(S^n) =0$ for all $n>1$. Also, we have that $\TC^\cD(S^1) =1$ as follows from $\TC^\cD(S^1) >0$ (since the circle is not simply connected) and
$\TC^\cD(S^1) \le \TC(S^1) =1$ (see Proposition \ref{prop:nonasph} below).
\end{example}

Next we compare $\TC^\cD(X)$ with $\TC(X)$.

\begin{proposition}\label{prop:nonasph} For a connected, locally path connected and semi-locally simply connected topological space $X$ one has
$$\TC^\cD(X) \leq \TC(X).$$

\end{proposition}

\begin{proof} The following argument appears in the proof of Theorem 4.1 of \cite{FTY}.
Let $\tilde X\to X$ be the universal cover of $X$. Consider the projection $Q: \tilde X\times_\pi \tilde X\to X\times X$ where $\pi=\pi_1(X)$ denotes the fundamental group of $X$ and
$\tilde X\times_\pi \tilde X$ stands for the quotient of $\tilde X\times \tilde X$ with respect to the diagonal action of $\pi$.
Clearly $Q$ is a covering map with the property that the image of the induced homomorphism $Q_\ast: \pi_1(\tilde X\times_\pi \tilde X) \to \pi_1(X\times X)$ is the diagonal.
Hence by Proposition \ref{cor:covspace} one has $\TC^\cD(X)=\secat (Q)$.

%

Define $p\colon X^I \to \tilde X\times_\pi \tilde X$ by $p(\gamma)=
[\tilde\gamma(0),\tilde\gamma(1)]$, where $\tilde\gamma\colon I \to \widetilde X$ is any
lift of the path $\gamma\colon I \to X$ and the brackets $[x, y]$ denote the orbit of the pair $(x, y)\in \tilde X\times\tilde X$ with respect to the diagonal action of $\pi$.
The map $p$ is well-defined although of course the lift $\tilde \gamma$ is not unique.
We obtain the
following commutative diagram.
$$\xymatrix{
X^I \ar[rr]^-p \ar[dr]_-{\pi_X} & & \tilde X\times_\pi \tilde X \ar[dl]^-{Q} \\
& X \times X &
}
$$
Clearly, a partial section $s\colon U \to X^I$ gives a partial section $\tilde s = p \,s\colon U \to
\tilde X\times_\pi \tilde X$, so we have
$$\TC(X) =\secat(\pi_X)\geq \secat(Q)=\TC^\cD(X).$$
\end{proof}

\begin{proposition}\label{aspher}
Let $X$ be an aspherical locally finite cell complex. Then $$\TC^\cD(X)=\TC(X).$$
\end{proposition}
\begin{proof}
Recall the known fact that an open subset of a locally finite cell complex is homotopy
equivalent to a countable cell complex. Indeed, by Theorem 1 of \cite{Mil}, a space is homotopy equivalent to a countable
cell complex if and only if it is homotopy equivalent to an absolute neighbourhood retract (ANR). Any locally finite cell complex is
an ANR and an open subset of an ANR is an ANR \cite{Han}. Thus, an open subset of a locally finite cell complex is an ANR and hence has
the homotopy type of a countable cell complex.

In view of Proposition \ref{prop:nonasph} we only need to establish the inequality $\TC(X)\le \TC^\cD(X)$.
Consider an open subset $U\subset X\times X$ such that the map induced by the inclusion $U\to X\times X$ on fundamental groups takes
values in a subgroup conjugate to the diagonal. Since $X\times X$ is aspherical and $U$ has the homotopy type of a cell complex, we see
that the inclusion $U\to X\times X$ is homotopic to a map with values in the diagonal $\Delta_X\subset X\times X$. Now we can use Lemma 4 from
\cite{FG} to conclude that a section of the path fibration $\pi_X: X^I \to X\times X$ over $U$ exists. The statement follows from the definitions.
\end{proof}

\begin{proposition}\label{prop4}
Let $f: X\to Y$ be a continuous map between path-connected topological spaces such that the induced map $f_\ast: \pi_1(X) \to \pi_1(Y)$ is an isomorphism. Then
$$\TC^\cD(X) \le \TC^\cD(Y).$$
\end{proposition}
\begin{proof} Let $U\subset Y\times Y$ be an open subset such that the induced homomorphism $\pi_1(U, u_i)\to \pi_1(Y\times Y, u_i)$
takes values in a subgroup conjugate to the diagonal $\Delta_Y$. Consider the preimage $V=(f\times f)^{-1}\subset X\times X$. The homomorphism
$\pi_1(V, v_i) \to \pi_1(X\times X, v_i)$ induced by the inclusion $V\to X\times X$ takes values in a subgroup conjugate to the diagonal
$\Delta_X$. Hence any open cover of $Y\times Y$ as in Definition \ref{def:tcd} defines a similar covering on $X\times X$ with the same number of sets.
\end{proof}

\begin{corollary}
$\TC^\cD(X)$ is a homotopy invariant of $X$.
\end{corollary}

We may therefore write $\TC^\cD(\pi)=\TC^\cD(K(\pi,1))$. Note that \propref{aspher} says that $\TC^\cD(\pi)=\TC(\pi)$.

\begin{corollary} Let $X$ be a path-connected cell complex with fundamental group $\pi=\pi_1(X)$.
Then
\begin{eqnarray}\label{leq}\TC^\cD(X) \le \TC^\cD(\pi).\end{eqnarray}
Moreover, if $\pi$ has cohomological dimension $\le 2$,
\begin{eqnarray}\label{equal}\TC^\cD(X) = \TC^\cD(\pi) .\end{eqnarray}
\end{corollary}


\begin{proof}
First note that we may construct the Eilenberg--Mac Lane complex $K=K(\pi, 1)$ starting from $X$ and attaching cells of dimension $\ge 3$.
We may apply Proposition \ref{prop4} to the inclusion $X\subset K$ which obviously induces an isomorphism of fundamental groups; this gives inequality (\ref{leq}).

To prove (\ref{equal}) we first convert the inclusion $X \hookrightarrow K$ into a fibration with fibre $F$ satisfying $\pi_i(F)=\pi_{i+1}(K,X)$. Note that, since
$K$ is aspherical and $\pi_1(X) \cong \pi_1(K)$, we have $\pi_i(F)=\pi_{i+1}(K,X)=0$ for $i=0,1$. Now, the obstructions to finding a section of the fibration
$X \to K$ lie in the groups (with local coefficients) $H^{i+1}(K;\pi_i(F))=H^{i+1}(\pi;\pi_i(F))$. But by what we have said above, these are trivial for $i=0,1$.
Furthermore, the hypothesis that $\chd(\pi) \leq 2$ implies that $H^{i+1}(\pi;\pi_i(F))=0$ for $i >1$. Hence, all obstructions vanish and there is a section
$K \to X$. In particular, the section is an isomorphism on $\pi_1$ since $X \to K$ is, so we simply apply \propref{prop4} to get
$$\TC^\cD(\pi)=\TC^\cD(K) \leq \TC^\cD(X).$$
Combining this with the first part then gives equality.
\end{proof}

As a generalisation of the previous Corollary we state:

\begin{lemma}
Let $X$ be a path-connected cell complex such that for some integer $k\ge 2$ the homotopy groups $\pi_j(X)$ are zero for all
$j$ satisfying $1<j<k$. If the cohomological dimension of $\pi=\pi_1(X)$ is at most $k$ then $ \TC^\cD(X) = \TC^\cD(\pi)$.
\end{lemma}

\begin{proof}
As above, the Eilenberg--Mac Lane space $K=K(\pi, 1)$ can be obtained from $X$ by attaching cells of dimension $k+1, k+2, \dots$.
We  have $X\subset K$ with $\pi_{i+1}(K,X)=0$ for $i=0,1,\ldots,k-1$ and we may again convert the inclusion to a fibration $X \to K$ with
fibre $F$ satisfying $\pi_i(F)=0$ for $i=0,1,\ldots,k-1$. The obstructions to finding a section of $X\to K$ lie in the groups
$H^{i+1}(\pi, \pi_{i}(F)) = H^{i+1}(K, \pi_{i}(F))$ and all these groups vanish because of our computation with $\pi_i(F)$ and
our assumption $\cd(\pi)\le k$. Finally we apply Proposition \ref{prop4} to achieve equality.
\end{proof}

\begin{example}
Let $X$ be a finite cell complex with fundamental group $\pi$. Suppose that $\cd(\pi)>2\,\dim X$.
Then we  have
$$\TC^\cD(X) \le \TC(X) \le 2\, \dim X < \cd(\pi) \le \TC^\cD(\pi).$$
Thus one may construct many examples with $\TC^\cD(X) < \TC^\cD(\pi)$. For instance, we may take a 2-dimensional
finite cell complex $X$ with fundamental group $\Z^5$ (i.e. the $2$-skeleton of $T^5$). Furthermore, since every finitely
presented group $\pi$ appears as the fundamental group of a closed $4$-manifold $X$, the gap between $\TC^\cD(X)$ and
$\TC^\cD(\pi)=\TC(\pi)$ can be as large as desired.
\end{example}

Let $X$ be a connected, locally path connected and semi-locally simply connected space with universal cover $P\colon \widetilde X \to X$.
The Lusternik-Schnirelmann
\emph{one-category}, $\cat_1(X)$, is defined as the sectional category $\secat(P)$  of $P$. This
interpretation of one-category goes back to Schwarz (\cite{Sv}) who also showed that $\cat_1(X) = \cat(f)$, where
$f\colon X \to K(\pi_1(X),1)$ classifies the universal cover and $\cat(f)$ is the category of the map $f$ (also see \cite{OS2}).
This latter description easily implies, for instance, that $\cat_1(T^n \times Y)=n$ when
$Y$ is simply connected. We describe now a relation between $\cat_1(X)$ and $\TC^\cD(X)$ that is akin to
that between $\cat(X)$ and $\TC(X)$.

\begin{proposition}\label{thm:TCcat}
If $X$ is a connected, locally path connected and semi-locally simply connected topological space then
\begin{eqnarray}\label{cat1}\cat_1(X) \leq \TC^\cD(X) \leq \cat_1(X \times X).\end{eqnarray}
\end{proposition}

\begin{proof}
Consider first the following general situation. Let $p:\tilde Z\to Z$ be a covering map with $\tilde Z$ connected and
$p_\ast(\pi_1(\tilde Z, \tilde z_0)) = H\subset \pi_1(Z, z_0)$. Let $f: A\to Z$ be an inclusion of a connected subspace. We obtain a pull-back diagram

$$
\xymatrix{
\tilde Z_A \ar[r]\ar[d]_-{p_A} &\tilde Z\ar[d]^-p\\
A\ar[r]^f_\subset & Z
}
$$
in which $\tilde Z_A$ can be identified with the preimage of $A$ under $p$ and $p_A$ is the restriction of $p$. Clearly the map $p_A$ is a covering map.
The set $\tilde Z_A$ is connected if and only if $f_*(\pi_1(A))$ and $p_*(\pi_1(\tilde Z))$ span $\pi_1(Z)$.
In that case $p_A$ is a connected covering corresponding to the subgroup $f_\ast^{-1}(H)\subset \pi_1(A)$.

Now consider the following diagram.
\begin{eqnarray}\label{diag1}
\xymatrix{
{\overline X} \ar[r] \ar[d]_{p'} & \widehat{X\times X} \ar[d]_q & \tilde X \times\tilde X \ar[l] \ar[d]^{P\times P}\\
X \ar[r]^-f_-\subset & X\times X & X\times X \ar[l]_-{=}
}
\end{eqnarray}
Here $P: \tilde X\to X$ is the universal cover of $X$ and $q: \widehat{X\times X} \to X\times X$ is the cover corresponding the diagonal subgroup $\Delta\subset \pi\times\pi$.
The map $f$ is an inclusion $f(x)=(x, x_0)$, where $x\in X$ and $x_0\in X$ is a base-point and $\overline X$ is the preimage $q^{-1}(f(X))$.
To apply the remark of the preceding paragraph, note that $f_*(\pi_1(X))$ and $q_*(\pi_1(\widehat{X \times X}))$ span $\pi_1(X \times X)$.
Hence it follows that $p': \overline X\to X$ is the universal cover of $X$.

Given an open subset $U\subset X\times X$ with a section $s: U\to \widehat{X\times X}$ we may restrict it to $f^{-1}(U)\subset X$
getting a section $s': f^{-1}(U)\to \overline X$. This shows that $\cat_1(X) = \secat (p')\le \secat (q)=\TC^\cD(X)$, thus proving the left inequality (\ref{cat1}).

Next we consider the right square of the diagram (\ref{diag1}). The map $P\times P$ is the universal covering and hence
$\secat(P\times P) =\cat_1(X\times X)\ge \secat(q) =\TC^\cD(X)$.
\end{proof}

%
%
%
%
%

Next we examine the case of real projective spaces.

\begin{proposition}
For any $n\ge 1$ one has
$$\TC^\cD(\RP^n) = \TC(\RP^n).$$
Hence, $\TC^\cD(\RP^n)=n$ for $n=1, 3, 7$ and for any $n\not=1, 3, 7$ the number $\TC^\cD(\RP^n)$ equals the smallest integer $k=k(n)$ such that
the projective space $\RP^n$ admits an immersion into $\R^k$.
\end{proposition}
\begin{proof} Assume first that $n\not=1, 3, 7$. In this case we may combine  Theorem 4.1, Corollary 4.4, Propositions 6.2 and 6.3 from \cite{FTY}
which imply our statement for $n\not=1,3, 7$. In the remaining case, i.e. when $n=1, 3, 7$, we know by \cite{FTY} that
$$\TC(\RP^n) = n=\cat(\RP^n)= \cat_1(\RP^n)\le \TC^\cD(\RP^n).$$
(Note that, in \cite{FTY}, the non-normalised convention $\TC(\ast)=1$ was used rather than the normalized convention $\TC(\ast)=0$ of this paper.)
The statement now follows from Proposition \ref{prop:nonasph}.
\end{proof}

\begin{remark}\label{rem:immerse}
In \cite[Theorem 4.5]{FTY} it is shown that for $n=2^k$, $\TC(\RP^n)=2^{k+1}-1$.  When $k=3$, for instance, we see that the lowest immersion
dimension for $\RP^8$ is $15$, so that $\TC^\cD(\RP^8)=15$, a result that would be difficult to obtain
directly from the definition.
\end{remark}

The next two results are analogues of results that equate $\TC$ to Lusternik-Schnirelman category for certain types of spaces. Although the first result
follows from the more general second, the relative simplicity of the proof in the presence of greater structure recommends its inclusion.

\begin{proposition}\label{group}
For any connected topological group $G$ one has $$\TC^D(G)=\cat_1(G).$$
\end{proposition}
\begin{proof}
Let $F: G\times G\to G$ be the map given by the formula $F(a, b) =ab^{-1}$. Denote $\pi=\pi_1(G, e)$ and consider the induced map on fundamental groups
$$\phi= F_\ast: \pi \times \pi = \pi_1(G\times G, e\times e) \to \pi.$$
We claim that the kernel of $\phi$ coincides with the diagonal subgroup $\Delta\subset \pi\times\pi$. Obviously the kernel of $\phi$ contains
$\Delta$. On the other hand, standard Eckmann-Hilton arguments used in the proof that $\pi_1(G)$ is abelian show that $F_*(a,b)=a-b$, so that
${\mathrm Ker}(\phi)=\Delta$.

This gives a pullback diagram of covering maps
$$
\xymatrix{
\widehat{G\times G} \ar[d]_q \ar[r]^{\tilde F} & \tilde G\ar[d]^P\\
G\times G \ar[r]^F & G
}
$$
where $P$ is the universal covering and $q$ is the covering corresponding to the diagonal subgroup.
From this diagram we obtain
$$\TC^\cD(G) =\secat(q) \le \secat(P) =\cat_1(G).$$
This complements the left inequality of Proposition \ref{thm:TCcat}.
\end{proof}

Next we give a generalisation of Proposition \ref{group}.

\begin{theorem}\label{thm:CWHspace}
 Let $X$ be a connected CW $H$-space. Then $$\TC^\cD(X)=\cat_1(X).$$
\end{theorem}

\begin{proof}
The proof given below is based on arguments used in \cite{LuptonScherer} to show that $\TC(X)=\cat(X)$ when $X$ is an $H$-space.

  Let $m:X\times X\to X$ denote the multiplication, which may be assumed to have a strict unit given by the base point $x_0\in X$.
  If $A$ is a based CW complex and $f,g:A\to X$ are based maps, their pointwise product $f\cdot g: A\to X$ is defined by
  $(f\cdot g)(a)=m(f(a),g(a))$ for all $a\in A$. By a theorem of James \cite{James}, the pointwise product endows the set of based
  homotopy classes $[A,X]$ with the structure of an \emph{algebraic loop}. In particular, equations of the form
  \[
  x\cdot a = b,\qquad a\cdot y = b,\qquad a,b\in [A,X]
  \]
  admit unique solutions $x,y\in [A,X]$.

  Let $p_1,p_2:X\times X\to X$ denote the coordinate projections. The loop structure on $[X\times X,X]$ guarantees the existence of
  a \emph{difference map} $D:X\times X\to X$ with the property that $p_1\cdot D \simeq p_2: X\times X\to X$.

   We claim that the induced homomorphism $D_*:\pi_1(X\times X)=\pi_1(X)\times\pi_1(X)\to \pi_1(X)$ on fundamental groups is given
   by $D_*(a,b)=b-a$ for all $a,b\in \pi_1(X)$.
  To see this, recall that the standard proof that $\pi_1(X)$ is abelian when $X$ is an $H$-space proceeds by showing that the two binary operations
  $+:\pi_1(X)\times\pi_1(X)\to \pi_1(X)$ and $\cdot:\pi_1(X)\times\pi_1(X)\to \pi_1(X)$, given respectively by concatenation and pointwise product
  of loops, share a two-sided identity and are mutually distributive. Therefore they agree. It follows that
  \begin{align*}
   a + D_*(a,b) & = (p_1)_*(a,b) + D_*(a,b) \\
   & = (p_1)_*(a,b)\cdot D_*(a,b) \\
    & = (p_1\cdot D)_*(a,b) \\
    & = (p_2)_*(a,b) \\
    & = b,
   \end{align*}
   which proves the claim.

   Now form the pullback of the universal cover $p:\widetilde{X}\to X$ along the map $D:X\times X\to X$ to obtain a covering $\rho:P\to X\times X$
   with $\secat(\rho)\leq\secat(p)=\cat_1(X)$. The image $\rho_*\pi_1(P)$ in $\pi_1(X\times X)$ is contained in the kernel of $D_*$, which as we
   have just seen equals the diagonal subgroup. Hence there is a lift $P\to \widehat{X \times X}$ of $\rho$ through the covering
   $q:\widehat{X \times X}\to X\times X$, which gives $\TC^\cD(X)=\secat(q)\leq\secat(\rho)$.

   Combining the two inequalities above, we have that $\TC^\cD(X)\leq \cat_1(X)$ when $X$ is an $H$-space.
   The opposite inequality is given by \propref{thm:TCcat}.
\end{proof}

In the somewhat overlooked paper \cite{Hand}, D. Handel shows that the free path fibration $X^I \to X \times X$ is a pullback of the based path
fibration $PX \to X$ over a certain map $h \colon X \times X \to X$ (which we describe below) if and only if $X$ is a CW H-space. This in itself implies
$\TC(X)=\cat(X)$ by standard inequalities, but it also can be used to give another proof of \thmref{thm:CWHspace}. Both proofs use special properties
of CW H-spaces, so it is worthwhile seeing each of them.

\begin{proof}[Alternative Proof of \thmref{thm:CWHspace}]
We use the notation of the proof of \thmref{thm:CWHspace}. The loop structure of $[X,X]$ shows that there is a right inverse map
$\eta\colon X \to X$ with the property that
$$X \stackrel{\Delta}{\to} X \times X \xrightarrow{1_X \times \eta} X \times X \stackrel{m}{\to} X$$
is nullhomotopic. Then we define $h=m(1_X \times \eta)$. On homology, this map induces $h_*(a,b)=a-b$ and since, for an H-space,
$\pi_1(X)=H_1(X)$, we see that $h_\#(a,b)=a-b$ as well. Then, since the diagonal $\Delta(\pi) \subset \pi \times \pi$ is the kernel of
this homomorphism, we can lift $h$ to $\tilde h\colon \widehat{X \times X} \to \tilde X$ as in
$$\xymatrix{
\widehat{X \times X} \ar[r]^-{\tilde h} \ar[d]_-q & \tilde X \ar[d]^-P \\
X \times X \ar[r]^-h & X .
}
$$
But the homomorphism $h_\#$ induces an isomorphism $(\pi \times \pi)\slash \Delta(\pi) \cong \pi$ and this in turn shows that
$\tilde h$ restricted to each fibre $(\pi \times \pi)\slash \Delta(\pi) \to \pi$ is a bijection. This then implies that the diagram above
is a pullback and the usual sectional category inequality gives
$$\TC^\cD(X) = \secat(q) \leq \secat(P) = \cat_1(X).$$
\end{proof}

\section{Fibration over a Covering and the Invariant $\widetilde \TC(X)$}\label{sec:fibcov}

Consider the situation
$$E\stackrel p\to \overline X \stackrel q \to X$$
where $p$ is a fibration with fibre $F$ and $q$ is a covering map where the space $\overline X$ is connected. The composition
$$q\circ p: E\to X$$
is a fibration with fibre $F'$ which is homeomorphic to the product $F\times F_0$ where $F_0$ is the fibre of $q$, i.e. a discrete set.


\begin{definition}\label{deftld} Define the number $\widetilde {\rm secat}(E\stackrel p\to \overline X\stackrel q\to X)$ as the minimal integer
$k\ge 0$ such that $X$ admits an open cover $X=U_0\cup \dots \cup U_k$ with the property that for each $i=0, 1, \dots, k$ the fibration
$p$ admits a continuous section over $q^{-1}(U_i)\subset \overline X$.
\end{definition}


We see immediately from the definition that
\begin{eqnarray}\label{tscsecat}
\tsc \ge \secat (p),
\end{eqnarray}
and
$\tsc =0$ if and only if the fibration $p: E\to \overline X$ admits a continuous section, i.e. when $\secat(p)=0$. Presently, we have no examples
where the inequality (\ref{tscsecat}) is strict. We then obtain the following estimate.

\begin{proposition}\label{prop1}
One has
\begin{eqnarray}\label{upper1}
\secat(q\circ p) \le \secat (q: \overline X \to X) + \tsc.
\end{eqnarray}
\end{proposition}

We postpone the proof  to recall some known results about open covers.
These results are described and proved in \cite{Dr1,Dr2}; the other relevant
references are \cite{O,Ha,Cu,OS1} as well as \cite[Exercise 1.12]{CLOT}.

An open cover $\W=\{W_0,\ldots,W_{m+k}\}$ of a space $X$ is an \emph{$(m+1)$-cover} if every subcollection
$\{W_{j_0}, W_{j_1} , \ldots, W_{j_m}\}$ of $m+1$ sets from $\U$ also covers $X$.
The following simple observation (see \cite{O} for instance) is often given without proof, but it is the basis
for many arguments in this approach.

\begin{lemma}\label{lem:mplus1cov}
A cover $\W=\{W_0,W_1, \ldots,W_{k+m}\}$ is an $(m+1)$-cover of $X$
if and only if each $x \in X$ is contained in at least $k+1$ sets of $\W$.
\end{lemma}

\begin{proof}
If $\W$ is an $(m+1)$-cover and
$x\in X$ is only in $k$ sets in $\W$, then  $k+m+1-k=m+1$ sets of the cover
do not contain $x$. These $m+1$ sets do not cover $X$,
contradicting the supposition on $\W$.

Suppose that each $x\in X$ is contained in at least $k+1$ sets
from $\W$ and choose a subcollection $\V$
of  $m+1$ sets from $\W$.  There
are only $k+m+1-(m+1)=k$ sets \emph{not}  in $\V$, so
$x$ must belong to at least
one set in $\V$.  Thus $\V$ covers $X$, and $\W$ is an $(m+1)$-cover.
\end{proof}

An open cover can be lengthened to a $(k+1)$-cover,
while retaining certain essential properties of the sets in the cover:

\begin{theorem}[{\cite{Cu,Dr1}}]\label{thm:extopencov}
Let $\U=\{U_0,\ldots,U_k\}$ be an open cover of a normal space $X$.
Then, for any $m=k,k+1,\ldots,\infty$, there is an open $(k+1)$-cover
of $X$, $\{U_0,\ldots,U_m\}$, extending $\U$ such that for $n > k$,
$U_n$ is a disjoint union of open sets that are subsets of the
$U_j$, $0 \leq j \leq k$.
\end{theorem}

In \thmref{thm:extopencov}, 
the sets $U_n$ possess
any properties of the original cover that are inherited by disjoint
unions and open subsets. In particular, if the cover $\U$ is categorical,
then the extended cover is also categorical. The following was proved
in \cite{OS1}. We recall the proof for the convenience of the reader.

\begin{lemma}\label{lem:propAB}
Let $X$ be a normal space with two open covers
\[
\U=\{U_0,U_1, \ldots,U_k\}\qquad
\mbox{and}
\qquad
\V=\{V_0,V_1, \ldots,V_m\}
\]
such that each set of $\U$ satisfies Property (A), and
each set of $\V$ satisfies Property (B).  Assume that
Properties (A) and (B) are inherited by   open
subsets and disjoint unions.  Then $X$ has an open cover
\[
\W=\{W_0, W_1 , \ldots,W_{k+m}\}
\]
by open sets satisfying \emph{both} Property (A) and Property (B).
\end{lemma}

\begin{proof}
Using Theorem \ref{thm:extopencov},
extend $\U$ to a $(k+1)$-cover $\widetilde\U=\{U_0,\ldots,U_{k+m}\}$ and
extend $\V$ to an $(m+1)$-cover $\widetilde\V=\{V_0,\ldots,V_{k+m}\}$.
Since each set in $\widetilde\U$ is a disjoint union of open subsets of
sets in $\U$, the cover $\widetilde \U$ consists of sets satisfying
Property (A); likewise,    each set in $\widetilde\V$ satisfies Property (B).
Since   Properties (A) and (B) are inherited by open subsets and disjoint
unions, we see that each set $U_i \cap V_j$ satisfies \emph{both} properties.

Therefore, the lemma will be proved if we can show that the
collection
$$
\W=\{ U_0 \cap V_0,U_1 \cap V_1 , \ldots, U_{k+m}\cap V_{k+m}\}.
$$
is an open cover of $X$. First, observe that since $\widetilde V$ is an $(m+1)$-cover,
each point  $x \in X$ lies in at least $k+1$ sets of $\widetilde \V$;
we may suppose, without loss of generality, that $x \in V_0 \cap \cdots \cap V_k$.
Next, since $\widetilde\U$ is a $(k+1)$-cover, the subcollection
$\{U_0,\ldots,U_k\}$ covers $X$, and so $x \in U_i$ for some $0\leq i \leq k$.
Thus $x\in U_i \cap V_i$ for at least one value of $i$ and $\W$ covers $X$.
\end{proof}

\begin{proof}[Proof of Proposition \ref{prop1}] We say that an open subset $U\subset X$ satisfies property A if $q$ has a section over $U$.
We say that an open subset $U\subset X$ satisfies property B if the fibration $p$ has a section over
$q^{-1}(U)\subset \overline X$. Both properties A and B are inherited by subsets and disjoint unions. If $k=\secat(q)$ and $l=\tsc$ then there exists
an open cover of $X$ by $k+1$ open sets satisfying A and there exists an open cover of $X$ by $l+1$ open subsets satisfying B. Hence by
Lemma \ref{lem:propAB} there is an open cover of $X$ by $k+l+1$ open subsets satisfying $A$ and $B$. If $U\subset X$ is an open subset
satisfying A and B then $q$ has a section over $U$ and
$p$ has a section over $q^{-1}(U)$. Then $q\circ p$ has a section over $U$. Hence $\secat (q\circ p) \le k+l$.
\end{proof}

\begin{example}\label{coincide} Let
$$E=\{(x, y, \omega); x, y\in \tilde X, \omega \in \tilde X^I, \omega(0)=x, \omega(1)=y\}/\pi=X^I,$$
where $\pi=\pi_1(X)$. Let $\overline X =\tilde X\times_\pi \tilde X$. Let
$p: E\to \overline X$ be given by $p([x, y, \omega]) =[x, y]$ and let $q: \overline X \to X\times X$ be given by
$q([x, y])=(Px, P y)$, where $P:\tilde X \to X$ is the universal cover. Now we have the situation
\begin{eqnarray}\label{pathtilde}
X^I\, \stackrel p \to \tilde X\times_\pi\tilde X \stackrel q \to X\times X.\end{eqnarray}
Obviously $\secat (q\circ p) =\TC(X)$ and $\secat(q) =\TC^\cD(X)$.
We shall introduce the shorthand notation
$$\widetilde \TC(X) = \widetilde{\secat}(X^I\stackrel p\to \tilde X\times_\pi\tilde X\stackrel q\to X\times X).$$
Inequality (\ref{upper1}) in this particular case becomes
\begin{eqnarray}\label{main1}
\TC(X)\le \TC^\cD(X) +\widetilde \TC(X).
\end{eqnarray}
This then establishes the first part of \introthmref{introthm:refine} of the Introduction.
\end{example}

\begin{remark}\label{rem:cattilde}
In \cite{Op} the invariant $\widetilde\cat(X)$ (called \emph{universal cover category}) was defined to be the least $k$ such that there exist open sets
$U_0, \ldots, U_k$ whose union covers $X$ and whose preimages $P^{-1}(U_j)$ under the universal covering map
$P\colon \tilde X \to X$ are contractible in $\tilde X$. It was then shown that the estimate
\begin{eqnarray}\label{eqn:cattilde}
\cat(X) \leq \cat_1(X) + \widetilde\cat(X)
\end{eqnarray}
holds. In light of \defref{deftld}, we see that (since the based path space $P\tilde X$ is contractible)
$$\widetilde\cat(X) = \widetilde\secat(P\tilde X \stackrel{p}{\to} \tilde X \stackrel{P}{\to} X)$$
with $\cat(X)=\secat(P\circ p)$ and $\cat_1(X)=\secat(P)$. Hence, by \propref{prop1},
(\ref{eqn:cattilde}) is a specialization of (\ref{upper1}).
\end{remark}

\begin{proposition}\label{prop:setc}
For any locally finite cell complex $X$, the number $$\widetilde \TC(X) = \widetilde{\secat}(X^I\, \stackrel p\to \tilde X\times_\pi\tilde X\stackrel q\to X\times X)$$
coincides with the strongly equivariant topological complexity $\TC_\pi^\ast(\tilde X)$ introduced by A. Dranishnikov \cite{Dr3}.
\end{proposition}

\begin{proof}
Recall that $\TC_\pi^\ast(\tilde X)$ is defined as the minimal number $k$ such that $\tilde X\times \tilde X$ can be covered by $k+1$
open sets $\tilde U_i$ such that each $\tilde U_i$  is $\pi\times\pi$-invariant and admits a $\pi$-equivariant continuous section
$\tilde s_i: \tilde U_i\to \tilde X^I$. Let $P: \tilde X\to X$ denote the universal covering projection.
Each $\pi\times\pi$-invariant open set $\tilde U_i\subset \tilde X\times \tilde X$ has the form
$(P\times P)^{-1}(U_i)$ where $U_i=(P\times P)(\tilde U_i)$ is an open subset of $X\times X$.

The definition of $\widetilde \TC(X)$ deals with open subsets $U_i\subset X\times X$ and continuous sections
$$s_i: V_i= q^{-1}(U_i) \to X^I.$$
If $\tilde U_i\subset \tilde X\times \tilde X$ denotes $(P\times P)^{-1}(U_i)$ then $V_i$ equals the quotient
$\tilde U_i/\pi$ with respect to the diagonal copy of $\pi\subset \pi\times \pi$.
We have the commutative diagram
\begin{eqnarray}\label{diagram}
\xymatrix{
\tilde X^I \ar[d]_-P \ar[r]^-{\tilde p} & \tilde X \times \tilde X \ar[d]_-Q  & \tilde U_i\ar[l]_-{\supset}\ar[d]^-{Q|_{V_i}}\\
X^I \ar[r]_-p &\tilde X\times_\pi \tilde X & V_i \ar[l]_-{\supset}
}
\end{eqnarray}
in which every vertical arrow is a principal $\pi$-bundle.

It is clear that every $\pi$-equivariant section $\tilde s: \tilde U_i\to \tilde X^I$  of the map $\tilde p\colon \tilde X^I \to \tilde X \times
\tilde X$ determines (by passing to $\pi$-orbits) the map
$s: V_i\to X^I$, and since $\tilde p\circ \tilde s$ is the inclusion $\tilde U_i\to \tilde X \times \tilde X$ we obtain that $p\circ s$ is the inclusion
$V_i\to \tilde X\times_\pi\tilde X$, i.e. $s$ is a section of $p$. Thus, $\TC_\pi^\ast(\tilde X) \geq \UTC(X)$.

For the converse, we want to show that each section $s: V_i\to X^I$ of $p$ determines a $\pi$-equivariant section of $\tilde p$.
To do so we need to recall a few basic facts about principal bundles, which can be found for example in \cite{Got,Husemoller}.

Each principal $\pi$-bundle $p: E\to B$ over a space with the homotopy type of a CW complex is classified by a homotopy class $\xi\in [B, B\pi]$.
Note that $E$ has a free $\pi$-action and $B=E/\pi$. If $p': E'\to B'$ is another principal $\pi$-bundle with class
$\xi'\in [B', B\pi]$ then a morphism of $\pi$-bundles
$$
\xymatrix{
E'\ar[d]_{p'} \ar[r]^-F & E\ar[d]_p\\
B'\ar[r]_-{f}& B
}
$$
exists if and only if $f^\ast(\xi) =\xi'$. Here the word morphism means that $F: E'\to E$ is a continuous map commuting with the $\pi$-action.
Note also that $F$ is uniquely determined by $f$ up to principal bundle equivalence, in the following sense: if $F_0, F_1: E' \to E$ are two $\pi$-maps with
$p\circ F_i= f\circ p'$ for $i=0, 1$ then
$$F_1= F_0\circ u $$
where $u:E'\to E'$ is a principal bundle equivalence of $p'$, that is, a $\pi$-homeomorphism which induces the identity on $B'$.

Let $s: V_i\to X^I$  be a continuous section of $p$.
Let $\xi\in [\tilde X\times_\pi\tilde X, B\pi]$ denote the class of the bundle $Q$, see diagram (\ref{diagram}). Then
$\xi|_{V_i} \in [V_i, B\pi]$
is the class of the bundle $Q|_{V_i}$ and
$\eta= p^\ast(\xi) \in [X^I, B\pi]$
is the class of the bundle $P$.
One has
$$s^\ast(\eta) =s^\ast p^\ast(\xi) = \xi |_{V_i}$$
and applying the general theory of principal bundles as described above we see that $s$ extends to a morphism of principal bundles
$$
\xymatrix{
\tilde X^I\ar[d]_-P & \tilde U_i\ar[l]_-{\tilde s}\ar[d]^-{Q|_{V_i}}\\
X^I & V_i\ar[l]^-s
}
$$
Note that $\tilde s: \tilde U_i \to \tilde X^I$ is a $\pi$-equivariant map, but need not be a section of $\tilde p$.

Consider the composition of morphisms of principal bundles
$$
\xymatrix{
\tilde X\times \tilde X \ar[d]_-{P\times P}& \tilde X^I\ar[d]_-P \ar[l]_-{\tilde p} & \tilde U_i\ar[l]_-{\tilde s}\ar[d]^-{Q|_{V_i}}\\
\tilde X\times_\pi\tilde X & X^I\ar[l]^-{p} & V_i\ar[l]^-s
}
$$
We know that the lower map $p\circ s: V_i\to \tilde X\times_\pi\tilde X$ is the inclusion. Let $j: \tilde U_i\to \tilde X\times\tilde X$ denote the inclusion, which is $\pi$-equivariant and covers $p\circ s$.  Using the uniqueness property of principal bundle maps described above, we see that $\tilde p \circ \tilde s\circ u= j$ for some principal bundle equivalence $u$ of $Q|_{V_i}$. Thus $\tilde s\circ u: \tilde U_i\to \tilde X^I$ is a $\pi$-equivariant section of $\tilde p$.
\end{proof}

\begin{corollary}\label{cor:univcov}
Let $P\colon \tilde X \to X$ be the universal covering projection with $\pi=\pi_1(X)$. Then
$$\UTC(X) \geq \TC(\tilde X).$$
\end{corollary}

\begin{proof}
Since the definition of strongly equivariant topological complexity just puts extra conditions on the usual $\TC$ diagram for $\tilde X$,
we have $\TC^*_\pi(\tilde X) \geq \TC(\tilde X)$. The result follows directly then from \propref{prop:setc}.
\end{proof}

\begin{lemma} \label{tcasph} A CW complex $X$ is aspherical if and only if $\widetilde \TC(X)=0$.
\end{lemma}
\begin{proof}
Suppose that $\widetilde \TC(X)=0$; that is, the fibration $p: X^I \to \tilde X\times_\pi\tilde X$ has a continuous section.
For $n\ge 2$ consider the composition
$$\pi_n(X) = \pi_n(X^I) \stackrel{p_\ast}\to \pi_n(\tilde X\times_\pi\tilde X)\stackrel{\simeq}\to \pi_n(X\times X) = \pi_n(X)\oplus \pi_n(X).$$
Since $p$ has a section this composition must be surjective. On the other hand, it is obvious that the image of this composition coincides with the
diagonal $\pi_n(X)\subset \pi_n(X)\oplus \pi_n(X)$. This is possible only when $\pi_n(X)=0$ for all $n\ge 2$.

On the other hand, suppose that $X$ is aspherical so that $\tilde X$ is contractible. The fibre of $p: X^I\to \tilde X\times_\pi\tilde X$ is $\Omega\tilde X$,
the loop space of the universal cover, so it is also contractible. This implies that $p$ has a section and hence
$\widetilde \TC(X) =0$.
\end{proof}

\begin{proposition}\label{thm:product}
Let $Z=X \times Y$ where $X=K(\pi,1)$ is aspherical and $Y$ is simply connected. Then
$$\TC^\cD(Z)=\TC(X)\quad \mbox{and} \quad
\UTC(Z)=\TC(Y).$$

\end{proposition}

\begin{proof} The first statement follows by applying Proposition \ref{prop4} to the
maps $X\to X\times Y\to X$ (injection and projection).

Now let's consider $\UTC(Z)$. The tower of fibrations (\ref{pathtilde}) looks in this case as follows
$$
\xymatrix{
X^I\times Y^I \ar[r]^-{p_X\times p_Y} & (\tilde X\times_\pi\tilde X)\times (Y\times Y)\ar[r]^-{q_X\times q_Y}&
(X\times X)\times (Y\times Y).}$$
As we mentioned in the proof of Lemma \ref{tcasph}, since $X$ is aspherical,
there exists a continuous section $\sigma: \tilde X\times_\pi\tilde X \to X^I$, i.e. $p_X\circ \sigma =1$.
Let $U\subset Y\times Y$ be an open set admitting a section $s: U\to Y^I$ of $p_Y$. Then the set
$(q_X\times q_Y)^{-1}(X\times X\times U) = \tilde X\times_\pi\tilde X\times U$ admits the section $\sigma\times s$ of
$p_X\times p_Y$. This shows that $\UTC(Z)\le \TC(Y).$

Conversely, assume that $V\subset (X\times X)\times (Y\times Y)$ is an open subset such that there is a continous section
$$\sigma: (q_X\times q_Y)^{-1}(V) \to X^I\times X^I.$$
Fix a point $x_0\in X$ and denote
$V'=V\cap (x_0\times x_0\times Y\times Y)\subset Y\times Y.$
Then clearly there exists a continuous section $\sigma': V'\to Y^I$. This shows that $\TC(Y)\le \UTC(Z).$
%
%
\end{proof}

Note that the inequality (\ref{main1})
reduces in this special case (i.e. when $Z=X\times Y$ with $X$ aspherical and $Y$ simply connected)
to the usual product inequality,
$$\TC(Z)\leq \TC^\cD(Z)+\UTC(Z)=\TC(X) + \TC(Y).$$

Finally, we give an example showing that \introthmref{introthm:refine} can be sharp and can be a better estimate than
that of \introthmref{introthm:Dranish} even when $\TC(\pi)$ is finite.

\begin{example}\label{exam:TCestimate}
{Consider the product $Z=T^2 \times S^2  $. Then $\TC^\cD(Z) =\TC(T^2)=2$ and $\UTC(Z) = \TC(S^2)=2$.
By applying the zero-divisors-cup-length estimate it is easy to see that $\TC(Z)\ge 4$ and
the inequality (\ref{main1}) gives in this case the inverse inequality $\TC(Z)\le 4$. Thus (\ref{main1}) is, in this instance, an equality.}
%
%
%
\end{example}

\section{A Connectivity-Dimensional Upper Bound}\label{sec:conndim}

In this section we establish an upper bound on $\tsc$ in terms of the dimension of $X$ and the connectivity of the fibre of $p$.

\begin{theorem}\label{main2} Let $X$ be a finite dimensional simplicial complex.
Consider two maps
$$E\stackrel p\to \overline X\stackrel q\to X,$$
where
$q:\overline X \to X$ is a covering map (not necessary regular) and  $p: E\to
\overline X$ is a fibration with $(k-1)$-connected fibre for some $k\ge 0.$
Then
\begin{eqnarray}\label{main3}\tsc \, \le\,  \left\lceil \frac{\dim (X)-k}{k+1}\right\rceil .\end{eqnarray}
\end{theorem}

\begin{proof} First we want to rephrase Definition \ref{deftld} as follows. We claim that the number $$\tsc$$ equals the smallest $c$ such
that there exists an increasing sequence of closed subsets
$$T_{-1}=\emptyset \subset T_0\subset T_1\subset T_2\subset \dots\subset T_c=X$$
such that for any $i=0, 1, \dots, c$ the fibration $p$ admits a continuous section over the set
$q^{-1}(T_i-T_{i-1})$. This follows by repeating the arguments of the proof of Proposition 4.12 from \cite{Inv}. (It can also be shown below
for skeleta using the fact that skeletal pairs are NDR pairs.)

Consider the sequence of skeleta
$$X^{(k)} \subset X^{(2k+1)}\subset X^{(3k+2)}\subset \dots \subset X^{((c+1)k+c)} = X$$
where $c$ is the smallest integer with $(c+1)k+c\ge \dim (X)$, i.e. $$c=\left\lceil \frac{\dim (X)-k}{k+1}\right\rceil.$$
Denoting $T_i=X^{((i+1)k+i)}$ for $i=0, 1,\dots, c$, we obtain an increasing sequence
of closed subsets $T_0\subset T_1\subset \dots \subset T_c=X$. We want to show that the fibration
$p: E \to \overline X$ admits a continuous section over the set $q^{-1}(T_{i}-T_{i-1})\subset \overline X$ for each $i$.

Next we make the following remark.
Let $p: E\to \overline X$ be a fibration with fibre $F$. Suppose that $F$ is $(k-1)$-connected. Let $A\subset \overline X$ be a subset which is
homotopy equivalent to a simplicial complex of dimension $\le k$. Then $p$ admits a section over $A$.
This follows directly by applying obstruction theory.

Applying the result of \S\ref{complements} we see that, denoting $T_i=X^{((i+1)k+i)}$,
the difference $T_i-T_{i-1}$ is homotopy equivalent to a simplicial complex of dimension $\le k$. Then its preimage
$q^{-1}(T_i-T_{i-1})$ is also homotopy equivalent to a simplicial complex of dimension $\le k$. Our statement now follows as we
assume that the fibre $F$ of $p: E\to \overline X$ is $(k-1)$-connected.
\end{proof}

Theorem \ref{main2} gives that $$\tsc \le \dim (X)$$ assuming the fibre of $p:E\to \overline X$ is not empty;
this is the case $k=0$. If the fibre $F$ is connected then
$$\tsc \le \left\lceil \frac{\dim (X)- 1}{2}\right\rceil,$$
and so on.

\begin{example}\label{exam:utcrpn}
We have that
$$\widetilde \TC(\RP^n) = \widetilde{\secat}((\RP^n)^I\stackrel p\to S^n\times_\pi S^n\stackrel q\to \RP^n\times \RP^n).$$
Because  $\RP^n\times \RP^n$ is covered by the $(n-1)$-connected space $S^n \times S^n$, \thmref{main2} gives
$$\UTC(\RP^n) \leq \left\lceil \frac{2\cdot n - (n-1)}{(n-1)+1}\right\rceil = 2.$$
If $n$ is even, then by \corref{cor:univcov}, we have $\UTC(\RP^n) \geq \TC(S^n)=2$, so $\UTC(\RP^n)=2$.
\end{example}

Finally we combine the inequality (\ref{main1}) with the upper bound of Theorem \ref{main2} in the situation of Example \ref{coincide}. 
We obtain the following result which is the second part of \thmref{introthm:refine} of the Introduction.

\begin{theorem}\label{main6}
Let $X$ be a finite dimensional simplicial complex.
Assume that the universal cover $\tilde X $ of $X$ is $k$-connected, i.e. $\pi_i(\tilde X) =0$ for $i\le k$.
Then
\begin{eqnarray}\label{main4}
\TC(X)\le \TC^\cD(X) + \left\lceil \frac{2\,\dim X -k}{k+1}\right\rceil.
\end{eqnarray}
\end{theorem}

In the special case $k=1$ (which is satisfied without extra assumptions of connectivity) the inequality (\ref{main4}) gives
\begin{eqnarray}\label{main5}
\TC(X)\le \TC^\cD(X) + \dim X\end{eqnarray}
which can be compared with the result of A. Dranishnikov \cite{Dr3}. The inequality (\ref{main5}) is stronger than (\cite{Dr3}) in cases when $\TC^\cD(X)< \TC(\pi)$.

As another special case of Theorem \ref{main6} we see that assuming that $\tilde X$ is 2-connected, one has
$$\TC(X)\le \TC^\cD(X) + \left\lceil \frac{2}{3}(\dim X-1)\right\rceil.$$

\section{Appendix: Complements}\label{complements}\label{sec:complements}

For convenience of the reader we state here a few known results which are used in the previous sections.

For a simplicial complex $K$ we denote by $V(K)$ the set of its vertices. The symbol $|K|$ denotes the geometric realisation of $K$.

\begin{lemma} Let $L\subset K$ be a simplicial subcomplex.
Suppose that $L$ has the following convexity property: every simplex of $K$ with all its vertices in $L$ lies in $L$.
Then the complement $|K|-|L|$ is homotopy equivalent to a simplicial complex $X$ with the vertex set
$V(X)= V(K)-V(L)$; a set of vertices in $V(X)$ forms a simplex in $X$ if and only if it forms a simplex in $K$.
\end{lemma}
\begin{proof}
For a vertex $v\in V(K)-V(L)$ we denote by $S_v\subset |K|$ its open star, i.e. the union of all open simplices containing $v$.
The family $$\{S_v; \, v\in V(K)-V(L)\}$$ is an open cover of the complement $|K|-|L|$ (here we use our assumption concerning the convexity of $L$).
The nerve of this open cover is isomorphic to $X$, as a simplicial complex. We observe that each
nonempty intersection $S_{v_0}\cap S_{v_1}\cap \dots \cap S_{v_k}$ is contractible since it coincides with the open star
of the simplex $(v_0, v_1, \dots, v_k)\in K$. Our statement now follows from Corollary 4G.3 on page 459 of Hatcher \cite{Hat}.
\end{proof}

Next we remove the convexity assumption:
\begin{lemma}
For any  simplicial subcomplex $L\subset K$, the complement $|K|-|L|$ is homotopy equivalent to the simplicial complex $Y$ with the vertex set labelled by the
set of simplices of $K$ which are not in $L$; simplices of $Y$ are in 1-1 correspondence with increasing chains
$\sigma_0\subsetneq\sigma_1\subsetneq \dots\subsetneq \sigma_k$ of simplices of $K-L$.
\end{lemma}
\begin{proof} Consider the barycentric subdivision $K'$ of $K$. Its vertices are labelled by the simplices of $K$; the simplices of $K'$ are labelled by the increasing chains
$\sigma_0\subsetneq\sigma_1\subsetneq \dots\subsetneq \sigma_k$
of simplices of $K$. The barycentric subdivision $L'$ of $L$ is a simplicial subcomplex of $K'$. The subcomplex $L'\subset K'$ has the convexity property: a
simplex of $K'$ belongs to $L'$ if and only if all its vertices lie in $L'$.
Now we simply apply the previous Lemma.
\end{proof}

\begin{corollary} For a $d$-dimensional simplicial complex $K$ and an integer $r< d$, the complement
$|K|-|K^{(r)}|$ has the homotopy type of a simplicial complex of dimension $\le d-r-1$.
\end{corollary}
\begin{proof}
Applying the previous Lemma we see that the complement $|K|-|K^{(r)}|$ has homotopy type of a simplicial complex with vertex set labelled by the set of simplices of $K$ having dimension $>r$ and with the simplices in 1-1 correspondence with the increasing chains of simplices of dimension $>r$. Since the length of any such chain is at most $d-r$ the result follows.
\end{proof}

This Corollary appeared in \cite{OS2} and also in \cite{Dr4}.

\section{Acknowledgements}
The authors would like to thank the Mathematisches Forschungsinstitut Oberwolfach for its generosity in supporting a July 2017 Research in Pairs stay where
this work was begun.

%
%

\end{document}